\documentclass[reqno, a4paper, 10pt]{amsart}
\usepackage{amssymb,url, color, mathrsfs, multirow, makecell, array, caption}
\usepackage[colorlinks=true, bookmarks=true, pdfstartview=FitH, pagebackref=true, linktocpage=true]{hyperref}
\usepackage[short,nodayofweek]{datetime}
\usepackage[pagewise,running,mathlines,displaymath, switch]{lineno}
\usepackage{times}
\usepackage{indentfirst, tabularx}
\usepackage[table]{xcolor}
\usepackage{float, hhline, colortbl}
\usepackage{graphicx}
\usepackage{longtable}

\newcolumntype{C}[1]{>{\centering\arraybackslash }b{#1}}

\hypersetup{
pdftitle={Exhaustive existence and non-existence results for some prototype polyharmonic equations},
pdfauthor={Quoc Anh Ngo, Van Hoang Nguyen, Quoc Hung Phan, Dong Ye},
colorlinks = true,
linkcolor = magenta,
citecolor = blue,
}

\theoremstyle{plain}
\numberwithin{equation}{section}
\newtheorem{theorem}{Theorem}[section]
\newtheorem{lemma}{Lemma}[section]

\newtheorem{proposition}{Proposition}[section]
\newtheorem{propABC}{Proposition}

\theoremstyle{remark}

\newtheorem{remark}{Remark}[section]

\DeclareMathOperator{\Rset}{\mathbb{R}}

\definecolor{brown}{rgb}{0.5,0,0}
\definecolor{backgroundcolor}{rgb}{0.98, 0.92, 0.73}

\newcommand{\ps}{p_{\mathsf S}}

\def\cfac#1{\ifmmode\setbox7\hbox{$\accent"5E#1$}\else\setbox7\hbox{\accent"5E#1}\penalty 10000\relax\fi\raise 1\ht7\hbox{\lower1.05ex\hbox to 1\wd7{\hss\accent"13\hss}}\penalty 10000\hskip-1\wd7\penalty 10000\box7 }

\author[Q.A. Ng\^o]{Qu\cfac oc Anh Ng\^o}
\address[Q.A. Ng\^o]{
Faculty of Mathematics-Mechanics-Informatics\\
VNU University of Science, Vi\^{e}t Nam National University\\
H\`{a} N\^{o}i, Vi\^{e}t Nam
--and--
Graduate School of Mathematical Sciences, The University of Tokyo, 3-8-1 Komaba, Meguro-ku, Tokyo 153-8914, Japan}
\email{\href{mailto: Q.A. Ng\^o <nqanh@vnu.edu.vn>}{nqanh@vnu.edu.vn}, \href{mailto: Q.A. Ng\^o <ngo@ms.u-tokyo.ac.jp>}{ngo@ms.u-tokyo.ac.jp}}

\author[V.H. Nguyen]{Van Hoang Nguyen}
\address[V.H. Nguyen]{Institute of Mathematics\\
Vietnam Academy of Science and Technology\\
Hanoi, Vietnam}

\email{\href{mailto: V.H. Nguyen <vanhoang0610@yahoo.com>}{vanhoang0610@yahoo.com}, \href{mailto: V.H. Nguyen <nvhoang@math.ac.vn>}{nvhoang@math.ac.vn}}

\author[Q.H. Phan]{Quoc Hung Phan}
\address[Q.H. Phan]{Institute of Research and Development\\
Duy Tan University\\
Da Nang, Vietnam}
\email{\href{mailto: Q.H. Phan <phanquochung@dtu.edu.vn>}{phanquochung@dtu.edu.vn}}
\usepackage{array}

\author[D. Ye]{Dong Ye}
\address[D. Ye]{Center for Partial Differential Equations, School of Mathematical Sciences and Shanghai Key Laboratory of PMMP, East China Normal
University, Shanghai 200062, China
--and--
IECL, UMR 7502, Universit\'e de Lorraine, 57073 Metz, France}
\email{\href{mailto: D. Ye <dye@math.ecnu.edu.cn>}{dye@math.ecnu.edu.cn}, \href{dong.ye@univ-lorraine.fr}{ dong.ye@univ-lorraine.fr}}

\makeatletter
\def\@cite#1#2{[\textbf{#1}\if@tempswa, #2\fi]}
\makeatother

\begin{document}

\allowdisplaybreaks

\title[Exhaustive existence and non-existence results for $\Delta^m u = \pm u^{\alpha}$ in $\Rset^n$]
{Exhaustive existence and non-existence results for some prototype polyharmonic equations in the whole space}

\begin{abstract}
In this paper, we are interested in entire, non-trivial, non-negative solutions and/or entire, positive solutions to the simplest models of polyharmonic equations with power-type nonlinearity
\[
\Delta^m u = \pm u^{\alpha} \quad \text{ in } \Rset^n
\]
with $n \geqslant 1$, $m \geqslant 1$, and $\alpha \in \Rset$. We aim to study the existence and non-existence of such classical solutions to the above equations in the full range of the constants $n$, $m$ and $\alpha$. Remarkably, we are able to provide necessary and sufficient conditions on the exponent $\alpha$ to guarantee the existence of such solutions in $\Rset^n$. Finally, we identify all the situations where any entire non-trivial, non-negative classical solution must be positive.
\end{abstract}

\date{\bf \today \; at \, \currenttime}



\maketitle

\section{Introduction}

In this paper, we are interested in the existence and non-existence results for the following equations
\begin{equation}
\label{eqMAIN}
\Delta^m u = \pm u^{\alpha} \quad \mbox{in the whole Euclidean space $\Rset^n$},
\end{equation}
where $n, m \geqslant 1$, and $\alpha \in \Rset$ is a parameter. It is easy to see that equations \eqref{eqMAIN} can be rewritten in the form
$$(-\Delta)^m u = \pm u^{\alpha}.$$
However, we intend to keep rather the notation $\Delta^m$ instead of $(-\Delta)^m$ for the convenience of presentation.

Among others, one basic reason that we are interested in such equations is that \eqref{eqMAIN} are the simplest models of polyharmonic equations with power-type nonlinearity. In the literature, equations of the form \eqref{eqMAIN} have attracted much attention in various mathematical directions, including the existence and non-existence results, the multiplicity, the regularity, the stability of solutions, the asymptotic behaviors at infinity of entire solutions, as well as Liouville-type results, etc. For interested readers, we refer to the monograph \cite{GGS10} for further motivations and results.

The exponent $\alpha$ here can take any value in $\Rset$. Regarding the nonlinearity $u^\alpha$, it is usually called \textit{superlinear}, \textit{sublinear}, or \textit{singular} respectively if $\alpha>1$, $\alpha\in (0,1)$, or $\alpha<0$. As we can imagine, the existence of solutions to \eqref{eqMAIN} strongly depends on the range of the exponent $\alpha$, on the dimension $n$, and on the fact that $m$ is even or odd.

It is well known that the semilinear polyharmonic equations arise in many physics phenomena. For example, several particular cases of \eqref{eqMAIN} have their origins such as the elasticity, the equilibrium states for thin films, the modeling of electrostatic actuations, etc.

The equations \eqref{eqMAIN} have also their root in conformal geometry. The equation
$$-\Delta u = k(x)u^\frac{n+2}{n-2}$$
with $n \geqslant 3$ is closely related to the famous Yamabe problem and the prescribing scalar curvature problem. The geometric aspect of higher order cases $m \geqslant 2$ are related the problem of prescribing $Q$-curvature on Riemannian manifolds. Loosely speaking, given a Riemannian manifold $({\mathcal M}^n,g)$ of dimension $n$, we denote by $P_m^g$ the GJMS operator of order $m$, constructed by Graham, Jenne, Mason, and Sparling in the celebrated work \cite{GJMS92}. The prescribing $Q$-curvature problem asks us to look for a positive solution $u$ to the following partial differential equation on ${\mathcal M}^n$
$$P_{2m}^g(u)= Q(x) u^\frac{n+2m}{n-2m}.$$
Under conformal projection or as the limit equation of blow-up analysis, we are often led to understand the problems like
$$(-\Delta)^m u = \pm u^\frac{n+2m}{n-2m}$$
in $\Rset^n$, so a special case of \eqref{eqMAIN}.

If the second order case $m=1$ is well understood, the situation of polyharmonic problems $m\geqslant 2$ is much less clear. For example, as far as we know, we cannot find exhaustive results on the existence or non-existence of positive solutions to \eqref{eqMAIN} for all exponents $\alpha \in \Rset$. To be clear, by solutions in this paper, we mean the classical solutions. Our main purpose here is to give a {\it complete} answer to this question, that is, to establish the existence or the non-existence of positive solutions to \eqref{eqMAIN} for any $m, n \geqslant 1$, and $\alpha \in \Rset$. We will handle also the case of non-trivial non-negative solutions when $\alpha \geqslant 0$, with the natural convention $0^0 = 1$.

\smallskip
In other words, we will find the {\it necessary and sufficient} conditions on the exponent $\alpha$ to confirm the non-existence, i.e.~the Liouville-type results for positive solutions with real exponents $\alpha$; and the Liouville-type results for non-negative solutions in $\Rset^n$ provided $\alpha \geqslant 0$. Such results are sometimes called \textit{optimal} Liouville-type theorems. The reason to consider separately the two classes of solutions is due to the lack of the strong maximum principle for high order elliptic operators, when $m \geqslant 2$.

In recent years, the Liouville property has emerged as an important subject in the analysis of nonlinear partial differential equations. In particular, Polacik, Quittner, and Souplet \cite{PQS07, PQS07b} developed a general method to derive universal pointwise estimates of local solutions from Liouville-type results. Their approach is based on rescaling arguments combined with a key doubling property, which is different from the classical rescaling method of Gidas and Spruck \cite{GS81b}. It turns out that one can obtain from Liouville-type theorems a variety of results on qualitative properties of solutions, such as {\it a priori} estimates, universal bounds, universal singularity and decay estimates, etc. For this reason, we expect to see many applications of Liouville-type theorems obtained here.

Before closing this section, we would like to mention the outline of the paper. The next section is devoted to the statement of our main results, which consist of two theorems. Theorem \ref{thmMinus} concerns the solvability of \eqref{eqMAIN} with a negative sign, that is $\Delta^m u = - u^{\alpha}$, while Theorem \ref{thmPlus} concerns solutions to $\Delta^m u = u^{\alpha}$. The proofs of Theorems \ref{thmMinus} and \ref{thmPlus} are presented in Section \ref{Proofs}, where we used several important approaches, including {\it a priori} integral estimates derived for local solutions, interpolation inequalities, the comparison principle for radial solutions, the derivation of sub/super polyharmonic properties and the Moser's iteration. In the last section, we identify all the situations where an entire non-trivial, non-negative solution must be positive, by proving Propositions \ref{thmCompareSetMinus} and \ref{thmCompareSetPlus}.



\section{Statement of main results}
\label{section-StatementOfResults}

\subsection{Some known results}

Let us start by reviewing some well-known results concerning the existence and non-existence of solutions to the problems \eqref{eqMAIN}. We recall the Sobolev exponent
\begin{align}\label{Sobolevexponent}
	\ps(m) =
	\begin{cases}
		\dfrac{n+2m}{n-2m} & \text{ if } n \geqslant 2m+1,\\
		\infty & \text{ if } n \leqslant 2m.
	\end{cases}
\end{align}

The first known result is for positive solutions to \eqref{eqMAIN} in the singular case $\alpha<0$.

\begin{propABC}\label{propA}
Let $m\geqslant 2$ be an integer and $n \geqslant 3$. Assume $\alpha<-1/(m-1)$. Then \eqref{eqMAIN} always possesses positive solutions.
\end{propABC}

Proposition \ref{propA} was proved by Kusano, Naito and Swanson via the Schauder--Tychonoff fixed-point theorem. More precisely, as a special case of Theorem 1 in \cite{KNS88}, if $n\geqslant 3$ and
\begin{align}\label{KNSest}
	\int_{0}^{\infty}t(1+t^{2m-2})^\alpha dt<\infty,
\end{align}
then \eqref{eqMAIN} possesses infinitely many positive radial solutions satisfying the following growth condition
\begin{align*}
	C_1(1+|x|^{2m-2})\leqslant u(x) \leqslant C_2(1+|x|^{2m-2}) \quad \mbox{in }\Rset^n,
\end{align*}
where $C_1$ and $C_2$ are positive constants. It is obvious that the integral in \eqref{KNSest} is finite when $\alpha<-1/(m-1)$.

\begin{remark}
We stress that the restriction on dimension $n\geqslant 3$ in Proposition \ref{propA}
is necessary for problem $\Delta^m u=-u^\alpha$; see the Proposition~\ref{pron=1,2} below for the non-existence result when $n\leqslant 2$.
\end{remark}

Next we collect some known results for the equation \eqref{eqMAIN} with a plus sign, namely
\begin{align}
\label{elliptic1}
(-\Delta)^{m} u= u^\alpha \quad \mbox{in }\Rset^n,
\end{align}
in the superlinear case $\alpha>1$. These results can be summarized as follows.

\begin{propABC}\label{propB}
Let $m$ be a positive integer. We have the following claims:
\begin{itemize}
 \item [(i)] If $1<\alpha<\ps (m)$ then the problem $(-\Delta)^{m} u=u^\alpha$ has no non-trivial non-negative solution in $\Rset^n$.
 \item [(ii)] If $n>2m$ and $\alpha \geqslant \ps(m)$, then the problem $(-\Delta)^{m} u=u^\alpha$ possesses positive radial solutions in $\Rset^n$.
\end{itemize}
\end{propABC}

Let us now comment on Proposition \ref{propB}. First, Part (i) is commonly known as the \textit{subcritical} case. From the definition of the Sobolev critical exponent $\ps (m)$ we know that $\ps (m)$ is finite if $n>2m$. In this setting, the second order case, namely $m=1$, was first established by Gidas and Spruck in \cite{GS81a} via the technique of nonlinear integral estimates and the Bochner formula. Chen and Li gave a different proof in \cite{CL} by using the moving plane method combined with the Kelvin transform. For higher order cases, Lin resolved in \cite{Lin98} the case $m=2$ and it was finally generalized by Wei and Xu in \cite{WX99} for any $m \geqslant 2$ via the argument of moving planes. For the remaining case $n\leqslant 2m$ and with arbitrary $m$, the non-existence result (i) can be deduced from the method of rescaled test-function \cite{MP01}; and also from the method of representation formula as presented in \cite{CAM08}.

Now we turn to the Part (ii). The case $m=1$ can be proved easily by applying the Pohozaev identity \cite{Poh65} with radial solutions on balls, or it
can be deduced from the shooting argument \cite{JL73}. In addition, if $\alpha = \ps(1)$, the \textit{critical} exponent, it was showed by Caffarelli, Gidas and Spruck in \cite{CGS89} that any positive solution to
$$-\Delta u = u^\frac{n+2}{n-2}\quad \mbox{in }\; \Rset^n$$
with $n \geqslant 3$ is radially symmetric, up to translations and dilations. This result was extended to the case of biharmonic equation by Lin \cite{Lin98} and to the general case $m \geqslant 2$ by Wei and Xu \cite{WX99}. More precisely, it was shown in \cite{WX99} that any positive solution $u$ to
\[
(-\Delta)^m u=u^\frac{n+2m}{n-2m} \quad \mbox{in }\; \Rset^n
\]
with $n > 2m \geqslant 2$ is of the following form
$$u(x)=\left(\frac{2\lambda}{1+\lambda^2|x-x_0|^2}\right)^\frac{n-2m}{2} \quad \mbox{ for some }\; x_0 \in \Rset^n, \; \lambda > 0.$$

Let us now turn to the \textit{supercritical} case, namely $\alpha > \ps(m)$. When $m=1$, the existence of positive solutions to \eqref{elliptic1} was obtained by Ni in \cite[Theorem 4.5]{Ni82}. For the higher order case, namely $m \geqslant 2$, the existence of positive solutions to \eqref{elliptic1} was shown by Liu, Guo and Zhang in \cite[Theorem 1.1]{LGZ06b}. They used a combination of the shooting method together with degree theory and the Pohozaev identity.

However, it becomes evident from the detailed description mentioned above that after putting together all the known results of the scientific literature the knowledge on this class of equations still appears quite fragmentary. Our aim in this paper is to consider all the situations $m \geqslant 2$, $n \geqslant 1$ and $\alpha \in \Rset$, and determine whether positive or non-negative solutions of \eqref{eqMAIN} exist.

For the sake of transparent presentation, we shall present our results in two different subcases according to the sign of the right-hand side. We will see that the range of $\alpha$ insuring the existence of solutions to equations \eqref{eqMAIN} strongly depends on the parity of $m$.

\subsection{Exhaustive results for $\Delta^m u=-u^\alpha$ in $\Rset^n$}
\label{subsec1}

As mentioned above, the results depend on the parity of $m$, it is more convenient to split the study for two equations:
\begin{align}\label{evenminus}
	\Delta^{2k} u=-u^\alpha \quad \mbox{in }\; \Rset^n \tag{P$_{2k}^-$}
\end{align}	
and
\begin{align}\label{oddminus}
	\Delta^{2k-1} u=-u^\alpha \quad \mbox{in }\; \Rset^n, \tag{P$_{2k-1}^-$}
\end{align}
where $k$ is a positive integer. For \eqref{evenminus}, as far as we know, there are many results which are limited to the case $k=1$, i.e.~for the biharmonic equation
\begin{align}\label{biharmonic}
	\Delta^2 u=-u^\alpha \quad \mbox{in }\; \Rset^n.
\end{align}
Here, the non-existence of positive solutions to \eqref{biharmonic} with $\alpha\in [-1,0]$ was first proved by Choi and Xu in \cite[Theorem 1.1]{CX09} for dimension $n=3$. This result was extended to all dimensions by Lai and Ye in \cite[Theorem 1.3]{LY16}. On the other hand, Proposition \ref{propA}, see also \cite[Theorem 3.1]{MR03}, ensures the existence of positive solutions for any $\alpha < -1$.

For the problem \eqref{oddminus}, when $k=1$, it is worth noticing that the class of positive solutions coincides with the one of non-trivial non-negative solutions, due to the strong maximum principle. For $k = 1$, the non-existence of positive solution when $\alpha\leqslant 1$ is well-known, see for instance the results in \cite[Theorem 2.7]{DM10} and \cite{AS11}; while the situations $\alpha > 1$ was fully settled in \cite{GS81a}. Recently, it was proved in \cite{DN17-arXiv} that
$$\Delta^3 u = -u^\alpha \quad \mbox{in }\; \Rset^3$$
has no positive solution if $\alpha \in [-1/2, 0)$.

We give here a complete answer to the question of existence for $\Delta^m u=-u^\alpha$. For convenience, we use the convention that $-1/(m-1)=-\infty$ when $m=1$. Our first result reads as follows:

\begin{theorem}\label{thmMinus}
Let $m$ be a positive integer. Then we have the following claims:
\begin{itemize}
\item [(i)] The problem $\Delta^m u=-u^\alpha$ possesses a positive solution if and only if either $n\geqslant 3$ and $\alpha < -1/(m-1)$ or $m$ is odd and $\alpha \geqslant \ps(m)$.
 \item [(ii)] The problem $\Delta^m u=-u^\alpha$ with $\alpha \geqslant 0$ has a non-trivial non-negative solution if and only if $m$ is odd and $\alpha \geqslant \ps(m)$.
\end{itemize}
\end{theorem}

The existence of positive or non-trivial non-negative solution to $\Delta^m u = -u^\alpha$ can be easily summarized in the following table.
\begin{center}
\begin{longtable}{
>{\centering\arraybackslash}p{.1\textwidth}|
>{\centering\arraybackslash}p{.185\textwidth}|
>{\centering\arraybackslash}p{.185\textwidth}|
>{\centering\arraybackslash}p{.185\textwidth}|
>{\centering\arraybackslash}p{.185\textwidth}
}
\hline
&
{\centering$\alpha < -\frac{1}{m-1} $}
	& $-\frac{1}{m-1}\leqslant \alpha < 0$
	& $0 \leqslant \alpha \leqslant 1$
	& $ \alpha > 1$ \\
\hline
\multirow{2}{.08\textwidth}{$u > 0$ $n\leqslant 2$} &  NO &  NO &  NO &  NO\\
& Prop. \ref{pron=1,2} & Prop. \ref{pron=1,2} & Prop. \ref{pron=1,2} & Prop. \ref{pron=1,2} \\
\hline
\multirow{2}{.08\textwidth}{$u > 0$ $n\geqslant 3$} & YES & NO &
\multirow{4}{.18\textwidth}{\centering NO \\ \medskip Props. \ref{pron=1,2} and \ref{proalpha>0}  } & \multirow{4}{.185\textwidth}{\centering YES \\ iff $m$ is odd and $\alpha \geqslant \ps(m)$ \\ \medskip Props. \ref{propB} and \ref{proalpha>1evenminus}} \\
& Prop. \ref{propA} & Prop. \ref{proalpha<0} &  \\
\hhline{---~|}
\multirow{2}{.08\textwidth}{$u \gneqq 0$} & \cellcolor{black!30} & \cellcolor{black!30} & & \\
& \cellcolor{black!30} & \cellcolor{black!30} & & \\
\hline
\caption{Existence results for problem $\Delta^m u =- u^\alpha$ in $\Rset^n$}
\label{tabevenminus}
\end{longtable}
\end{center}

Recall that we are concerned with classical solutions, then for $\alpha < 0$, there is no {\sl proper} non-negative solution, out of positive solutions. That's the reason of the above grey cells.

As the equation \eqref{oddminus} with $\alpha \geqslant \ps(m)$ always admits positive solutions, hence non-trivial, non-negative solution. A natural question for \eqref{oddminus} is that whether or not there is non-trivial, non-negative but {\it not positive} solution, the following maximum type result indicates that such solution does not exist.
\begin{proposition}\label{thmCompareSetMinus}
Let $m \geqslant 1$ and $\alpha > 1$. Then any non-trivial, non-negative solution to the equation $\Delta^m u =- u^\alpha$ in $\Rset^n$ must be positive everywhere.
\end{proposition}

Clearly, our contributions in Theorem~\ref{thmMinus} are multifold:
\begin{itemize}
\item By determining the sign of $\Delta^{m-1} u$ with Lemma \ref{prow0}, we show quickly the non-existence of positive solution to $\Delta^m u = -u^\alpha$ in $\Rset^n$ with $n = 1, 2$ for any $\alpha \in \Rset$; see Proposition \ref{pron=1,2}.

\item For any $m\geqslant 1$ and $n \geqslant 3$, we obtain the non-existence of positive solution in the range $\alpha\in [-1/(m-1), 1]$; see Propositions \ref{proalpha<0} and \ref{proalpha>0}.

\item We obtain the non-existence of non-trivial, non-negative solution in the range $\alpha\in [0, \infty)$ for the equation \eqref{evenminus}; see Propositions~\ref{proalpha>0} and \ref{proalpha>1evenminus}.
\end{itemize}

\smallskip
The proof of the non-existence in the singular case $\alpha\in[-1/(m-1), 0)$ with $n\geqslant 3$ relies on the convexity of the function $t \mapsto t^\alpha$ and comparison principle. In the superlinear case $\alpha>1$ and for the equation \eqref{evenminus}, we made use of the integral estimate and a Liouville type result; see Lemma \ref{lem_general_Liouville}. However, the case $\alpha\in [0,1]$ is significantly more delicate, it seems that many well-known approaches, such as the standard rescaled test-function method \cite{MP01}, the moving plane technique \cite{CL}, the argument of maximum principle \cite{AS11}, the representation formula method \cite{CAM08}, or the derivation technique of super/sub polyharmonic property \cite{Lin98, WX99, CL13, Ngo17}, cannot be applicable in general. Our proof in the sublinear case is inspired by the idea of Serrin and Zou \cite{SZ96} and Souplet \cite{Sou09}, it is based on the integral estimates and Moser's iteration method; see the proof of Propositions~\ref{proalpha>0}.



\subsection{Exhaustive results for $\Delta^m u=u^\alpha$ in $\Rset^n$}
\label{subsec3}
We give here a complete answer to the question of existence for \eqref{eqMAIN} with the plus sign. Our second result reads as follows:
\begin{theorem}\label{thmPlus}
	Let $m$ be a positive integer. Then we have the following claims.
\begin{itemize}
 \item [(i)] The problem $\Delta^m u=u^\alpha$ possesses positive solutions if and only if either $\alpha \leqslant 1$ or $m$ is even and $\alpha\geqslant \ps(m)$.
 \item [(ii)] The problem $\Delta^m u=u^\alpha$ with $\alpha \geqslant 0$ has non-trivial non-negative solutions if and only if either $0\leqslant \alpha \leqslant 1$ or $m$ is even and $\alpha\geqslant \ps(m)$.
\end{itemize}
\end{theorem}

The results of Theorem \ref{thmPlus} are summarized in the following table.
\begin{center}
\begin{longtable}{
>{\centering\arraybackslash}p{.1\textwidth}|
>{\centering\arraybackslash}p{.256\textwidth}|
>{\centering\arraybackslash}p{.256\textwidth}|
>{\centering\arraybackslash}p{.256\textwidth}
}
\hline
&
{\centering$\alpha < 0$}
	& $0 \leqslant \alpha \leqslant 1$
	& $1<\alpha$ \\
\hline
\multirow{2}{.08\textwidth}{$u > 0$} & YES &
\multirow{4}{.18\textwidth}{\centering YES \\ \medskip Prop. \ref{proalpha<1plus}  } & \multirow{4}{.25\textwidth}{\centering YES \\ iff $m$ is even and $\alpha \geqslant \ps(m)$ \\ \medskip Props. \ref{propB} and \ref{proalpha>1evenminus}} \\
& Prop. \ref{proalpha<1plus} &    \\
\hhline{--~|}
\multirow{2}{.08\textwidth}{$u \gneqq 0$} & \cellcolor{black!30} & & \\
& \cellcolor{black!30}  & & \\
\hline
\caption{Existence results for the problem $\Delta^m u = u^\alpha$ in $\Rset^n$}
\label{tabevenplus}
\end{longtable}
\end{center}

As before, we will split our study into two equations according to the parity of $m$, that is,
\begin{align}\label{evenplus}
	\Delta^{2k} u=u^\alpha \quad \mbox{in } \; \Rset^n \tag{P$_{2k}^+$}
\end{align}	
and
\begin{align}\label{oddplus}
	\Delta^{2k-1} u=u^\alpha \quad \mbox{in } \; \Rset^n, \tag{P$_{2k-1}^+$}
\end{align}	
where $k$ is a positive integer. Our contribution are twofold here:
\begin{itemize}
\item We give a unified proof of the existence of positive solutions for all $\alpha\leqslant 1$; see Proposition \ref{proalpha<1plus}.
\item We prove the non-existence of non-trivial, non-negative solutions for \eqref{oddplus} with any $\alpha > 1$; see Proposition \ref{proalpha>1evenminus}.
\end{itemize}

In the second order case, it is well-known that the problem $\Delta u=u^\alpha$ in $\Rset^n$ has no positive solution if $\alpha>1$, but it possesses a positive one if $\alpha\leqslant 1$. More precisely, the non-existence for the superlinear case $\alpha>1$ is a consequence of the so-called Keller--Osserman criteria developed by Keller \cite{Kel57} and Osserman \cite{Oss57}. Their theory can be employed to show that the equation $\Delta u = u^\alpha$ admits no non-trivial, non-negative, entire solution whenever $\alpha > 1$, see also \cite[Lemma 2]{Bre84}. When $\alpha \leqslant 1$, the existence of radial solutions can be easily obtained by the monotonicity of $u(r)$. We note that for $m \geqslant 2$, we can also apply Proposition \ref{propA} to obtain the existence of solutions to \eqref{evenplus} and \eqref{evenminus} for $\alpha < -1/(m-1)$.

Similarly to the question raised for \eqref{oddminus}, we can ask here if the set of non-trivial, non-negative solutions and the set of positive solutions coincide. We provide a complete answer as follows.

\begin{proposition}\label{thmCompareSetPlus}
Let $m$ be a positive integer and $\alpha \geqslant 0$. Then the equation $\Delta^m u = u^\alpha$ possesses entire, non-trivial, non-negative but {\it not strictly positive} solution in $\Rset^n$  if and only if $$\alpha \in [0, 1] \quad \mbox{and}\quad (\alpha, m) \ne (1,1).$$
In other words, if either $\alpha > 1$ or $(\alpha, m) = (1,1)$, then any entire, non-trivial, non-negative solution to $\Delta^m u = u^\alpha$ must be positive everywhere.
\end{proposition}

Before closing this section, we would like to comment on Propositions \ref{thmCompareSetMinus} and \ref{thmCompareSetPlus}. From our point of view, these results can be regarded as maximum principle results. As far as we know, similar results do exist in the literature, however, with some limitations, see for example \cite{CAM08}.



\section{Preliminaries}

In what follows, the notation $\Delta^i u$ stands for $u$ when $i=0$. The notation $B_r$ is always understood as the open ball $B_r(0)$ centered at the origin with radius $r$. Also, we denote always by $\overline u (r)$ the spherical average of $u$ centered at the origin on the sphere $\partial B_r$, the boundary of the ball $B_r$, that is
\[
\overline u (r) = \frac{1}{|\partial B_r|}\int_{\partial B_r} u d\sigma.
\]
When $u$ is a radial function, we also use the notation $u(r)$. Throughout the paper, the symbol $C$ denotes a generic positive constant whose value could be different from one line to another.

Here are some basic results, which will be useful for our analysis.

\begin{lemma}
\label{comparison}
Let $m \geqslant 1$ and $v_1, v_2: B_R \to I \subset \Rset$ be two $C^{2m}$ radial functions verifying
$$\Delta^m v_1 \geqslant f(v_1), \quad \Delta^m v_2 \leqslant f(v_2) \quad \mbox{in }\; B_R$$
and
$$\Delta^iv_1(0) \geqslant \Delta^iv_2(0), \quad \forall\; 0\leqslant i \leqslant m-1.$$
If $f$ is non-decreasing in $I$, then $v_1 \geqslant v_2$ in $B_R$. In other words, $v_1(r) \geqslant v_2(r)$ for all $r \in [0, R)$.
\end{lemma}

The above comparison principle is a special form of more general well-known results; see for instance \cite[Proposition A.2]{FF16} or \cite[Remark 2.3]{LS09}. An easy consequence of the above comparison principle is the following pointwise estimate.

\begin{lemma}\label{lemprioriestimate}
Let $u$ be in $C^{2m}(\Rset^n)$ satisfying $\Delta^m u\leqslant 0$ in $\Rset^n$, then we have
\begin{equation}\label{wabove}
\overline u (r) \leqslant u (0) + \sum_{i=1}^{m-1} \frac{ \Delta^i u (0) r^{2i} }{\prod_{1 \leqslant k \leqslant i} \left[2k(n+2k-2)\right]}, \quad \forall\; r \geqslant 0.
\end{equation}
\end{lemma}

\begin{proof}
Let $\Phi$ be the radial function defined by
$$\Phi(r):= u (0) + \sum_{i=1}^{m-1} \frac{ \Delta^i u(0) r^{2i} }{\prod_{1 \leqslant k \leqslant i} \left[2k(n+2k-2)\right]}.$$
There hold then
$$\Delta^m \Phi \equiv 0 \geqslant \overline{\Delta^m u} = \Delta^m \overline u \quad \mbox{in } \; \Rset^n$$
and $\Delta^{i}\Phi(0) = \Delta^i u (0) = \Delta^i \overline{u}(0)$ for any $0\leqslant i \leqslant m-1$. Applying Lemma \ref{comparison} with $f\equiv 0$, there holds $\overline u \leqslant \Phi$ in $\Rset^n$.
\end{proof}

By an elementary computation involving the Gamma function, it is easy to verify that
\[
\prod_{1 \leqslant k \leqslant p} \left[2k(n+2k-2)\right]
=2^{2p} p! \Gamma \Big(p + \frac n2 \Big)/\Gamma \Big(\frac n2 \Big), \quad \forall\; p, n \in {\mathbb N}^* .
\]
Therefore, the right hand side of \eqref{wabove} is nothing but the main part of classical Pizzetti's expansion formula in \cite{Piz09}; see also Equation (8) in Nicolesco's paper \cite{Nic32}. In Pizzetti's formula, there is a last term involving $\Delta^m u$. We can remark that if $\Delta^m u\leqslant 0$ in $\Rset^n$, then the remained term in Pizzetti's formula is non-positive, which implies then \eqref{wabove}. Nevertheless, our proof of \eqref{wabove} is simple and constructive.

The following result is a simple but important fact of our approach.

\begin{lemma}\label{prow0}
	Let $m\geqslant 1$. Then we have the following claims:
\begin{itemize}	
 \item [(i)] If $u$ be a positive function satisfying $\Delta^m u<0$ in $\Rset^n$, then $\Delta^{m-1} u>0$ in $\Rset^n$.
	
 \item [(ii)] If $u$ be a non-negative function satisfying $\Delta^m u\leqslant 0$ in $\Rset^n$, then $\Delta^{m-1} u\geqslant 0$ in $\Rset^n$.
\end{itemize}
\end{lemma}

\begin{proof}
Consider first the claim (ii). Set $w = \Delta^{m-1}u$, suppose that there was a point $x_0 \in \Rset^n$ such that $w(x_0) <0$. By translation, we may assume that $x_0=0$.
Moreover, it follows from Lemma~\ref{lemprioriestimate} that $u$ satisfies the estimate \eqref{wabove}. As $\Delta^{m-1}u(0)<0$, there holds $\overline u(r) <0$ for $r$ large enough. This is impossible because $u$ is non-negative in $\Rset^n$. The point (ii) holds true.

Now we consider (i). Set again $w = \Delta^{m-1}u$, we have $w \geqslant 0$ in $\Rset^n$ by (ii). If $w$ vanishes at some $x_1 \in \Rset^n$, then $w$ attains its minimum at $x_1$. However, this contradicts the fact that $\Delta w(x_1)<0$, we are done.
\end{proof}

It is worth noting that without the non-negativity of $u$, in general, the result of Lemma \ref{prow0} does not hold. For example, it was shown in \cite[Lemma 7.8]{FF16} that there are infinitely many entire radial solutions to $\Delta^{2k+1} u = -e^{u}$ for which $\Delta^{2k} u$ changes sign.

The following Liouville type result is a crucial step in the proof of Proposition~\ref{proalpha>1evenminus}.

\begin{lemma}\label{lem_general_Liouville}
Assume that $u$ is a $C^{2m}$ non-negative function in $\Rset^n$, verifying $(-\Delta)^m u \leqslant 0$ in $\Rset^n$ and
\begin{align}\label{test_initial}\int_{B_R} u dx = o(R^n), \quad \mbox{as } \; R \to \infty.
\end{align}
Then $u \equiv 0$ in $\Rset^n$.
\end{lemma}

\begin{proof}
 Let $v_k=(-\Delta)^k u$, for $0\leqslant k \leqslant m$. We shall prove by backward induction that
\begin{align}\label{vh}
v_k\leqslant 0 \;\; \mbox{in }\;\Rset^n
\end{align}	
for $k = m, m-1, \ldots, 0$. It is obvious that \eqref{vh} is true for $k=m$. Suppose now \eqref{vh} is true for $j+1 \leqslant k \leqslant m$ with some $j\geqslant 0$, we shall show that $v_j\leqslant 0$ in $\Rset^n$. We have two possible cases.

\noindent{\sl Case 1: $j$ is odd}. As $\Delta^{j+1}u = v_{j+1}\leqslant 0$, Lemma~\ref{prow0}(ii) gives $v_j\leqslant 0$.

\noindent{\sl Case 2: $j$ is even}. We will prove $v_j\leqslant 0$. By way of contradiction, assume that there exists some $x_0 \in \Rset^n$ such that $v_j(x_0)>0$. Up to a translation, we may further assume that $x_0=0$. Then $$\Delta \overline v_j = \overline{\Delta v_j} = -\overline v_{j+1} \geqslant 0 \quad \mbox{in }\; \Rset^n.$$
We have then $\overline v_j '(r) \geqslant 0$ for any $r \geqslant 0$, hence $\overline v_j(r) \geqslant \overline v_j(0) = v_j (0) >0.$ Let $\psi$ be a smooth, radial, cut-off function satisfying $0\leqslant \psi\leqslant 1$ and
\begin{align}
\label{psi}
\psi(x)=
\begin{cases}
0 &\text{ if } |x|\geqslant 2,\\
1 &\text{ if } |x|\leqslant 1.
\end{cases}
\end{align}
For any $R>0$, we can estimate
\begin{equation}\label{test4}
\begin{aligned}
\int_{\Rset^n}v_j(x)\psi\left(\frac{x}{R}\right) dx & \geqslant \int_{B_R}v_j(x) dx \\
& = |\mathbb S^{n-1}| \int_0^R \overline v_j(r) r^{n-1} dr
\geqslant \frac 1n  |\mathbb S^{n-1}| \overline{v}_j(0)R^n .
\end{aligned}
\end{equation}
On the other hand, there holds
\begin{align}
\label{test5}
\begin{split}
\int_{\Rset^n}v_j(x)\psi\left(\frac{x}{R}\right) dx & = \int_{\Rset^n}(-\Delta)^j u(x)\psi\left(\frac{x}{R}\right) dx\\ &= R^{-2j}\int_{\Rset^n}\ u(x)(-\Delta)^j\psi\left(\frac{x}{R}\right) dx\\
& \leqslant CR^{-2j} \int_{B_{2R}}u(x) dx.
\end{split}
\end{align}
Putting \eqref{test4} and \eqref{test5} together gives
\begin{align*}
0 < \overline{v}_j(0)\leqslant CR^{-2j-n} \int_{B_{2R}}u(x) dx.
\end{align*}
Letting $R\to \infty$ and using \eqref{test_initial} we meet a contradiction. We get then $v_j\leqslant 0$ in $\Rset^n$. Therefore, by induction principle, \eqref{vh} is true as claimed. Taking $k=0$ in \eqref{vh}, we have $u \leqslant 0$ in $\Rset^n$, hence $u\equiv 0$ in $\Rset^n$.
\end{proof}

The last result in this subsection is a classical interpolation-type estimate, which plays an important role in our proof of the non-existence result for $\Delta^m u = - u^\alpha$ with $0<\alpha<1$; see the proof of Proposition \ref{proalpha>0}.

\begin{lemma}\label{lem6}
	Let $m$ be a positive integer. Let $z$ be a function in $W^{2m,\ell}(B_{2R})$ for some $\ell>1$. Then for any exponent $q>1$ such that
\begin{align}\label{exponentq}
	\frac{1}{q}\geqslant \frac{1}{\ell}-\frac{2m}{n},
\end{align}
there holds
\begin{align*}
	\Big(\int_{B_R}z^{q}dx\Big)^{1/q}&\leqslant CR^{\frac{n}{q}+2m-\frac{n}{\ell}}\Big(\int_{B_{2R}}|\Delta ^mz|^{\ell}dx\Big)^{1/\ell} +CR^{\frac{n}{q}-n}\int_{B_{2R}}zdx,
\end{align*}
where $C=C(m, n,\ell,q)$.
\end{lemma}

\begin{proof}
By the dilation $w(x)=z(Rx)$, we obtain
\[
\int_{B_R}z^{q}dx = R^{n} \int_{B_1} w^{q}dx, \quad \int_{B_{2R}}zdx= R^n \int_{B_{2}} w dx,
\]
and
\[
\int_{B_{2R}}|\Delta ^mz|^{\ell}dx= R^{-2m \ell + n} \int_{B_2}|\Delta ^m w|^{\ell}dx.
\]
From these identities, the desired inequality is equivalent to
\begin{align*}
	\|w\|_{L^q(B_1)}\leqslant C\|\Delta^m w\|_{L^\ell(B_2)}+ C \|w\|_{L^1(B_2)}
\end{align*}
for $w\in W^{2m,\ell}(B_{2})$. However, this follows from \eqref{exponentq} and standard elliptic estimate; see for instance \cite[Theorem 2.20]{GGS10}. The lemma is proved.
\end{proof}	

\section{Proof of the main results}
\label{Proofs}

This section is devoted to the proof of our main results. We prove some Liouville type results in subsections \ref{subsec-ProofWithMinusSign} and \ref{subnew1}, while some existence results are proved in subsection \ref{subnew2}. It is worth noticing that for each case in Tables \ref{tabevenminus} and \ref{tabevenplus} above, we have already included the name of the main proposition yielding the result in the case. Therefore, there is no need to write a proof for Theorems \ref{thmMinus} and \ref{thmPlus}.

\subsection{Non-existence results for $\Delta^m u=-u^\alpha$}
\label{subsec-ProofWithMinusSign}

This subsection is devoted to the non-existence results in Theorem \ref{thmMinus}, and we do not consider specially the situations under applications of Propositions \ref{propA} and \ref{propB}.

\subsubsection{For dimensions 1 and 2}
\label{subsubsec-MinusSign-NE-n=1,2}

Let us start with the case $n \leqslant 2$ and this corresponds to the second and fifth rows in Table \ref{tabevenminus}. We will prove that in dimension one and two, the equation
\begin{align}
\label{negative}
\Delta^m u = -u^\alpha \quad \mbox{in }\; \Rset^n
\end{align}
has no positive solution for any $\alpha \in \Rset$; and has no non-trivial non-negative solution for any $\alpha \geqslant 0$. In fact, these claims are trivial consequence of the following result.

\begin{proposition}\label{pron=1,2}
Let $m$ be a positive integer and $n \leqslant 2$. If $u$ is a non-negative $C^{2m}$ function verifying $\Delta^m u \leqslant 0$ in $\Rset^n$, then $\Delta^m u \equiv 0$ in $\Rset^n$.
\end{proposition}

\begin{proof}
As $\Delta^m u \leqslant 0$ in $\Rset^n$, Lemma~\ref{prow0}(ii) shows that $\Delta^{m-1} u =: w \geqslant 0$. This means that $w$ is a non-negative, super-harmonic function in $\Rset^n$. It is well-known that such $w$ must be constant in $\Rset^n$ if $n \leqslant 2$; see \cite[Theorem 3.1]{Far07} for the case $n=2$. Hence, $\Delta^m u = \Delta w = 0$ everywhere.
\end{proof}

\subsubsection{For negative values of $\alpha$}

Here we prove the non-existence of positive solution to \eqref{negative} for $n \geqslant 3$ and suitable negative values of $\alpha$.

\begin{proposition}\label{proalpha<0}
Let $n \geqslant 3$. Then the equation \eqref{negative} has no positive solution for any $\alpha \in [-1/(m-1), 0)$ if $m > 1$ or for any $\alpha < 0$ if $m = 1$.
\end{proposition}

\begin{proof}
Assume that $n\geqslant 3$, $m \geqslant 1$, and $\alpha \in [-1/(m-1), 0]\cap \Rset$. By way of contradiction, suppose that $u$ is a positive solution to \eqref{negative}. Using Lemma~\ref{lemprioriestimate}, we have
\[
\overline u (r) \leqslant u (0) + \sum_{i=1}^{m-1} \frac{ \Delta^i u(0) r^{2l} }{\prod_{1 \leqslant k \leqslant i} [2k(n+ 2k - 2)] }, \quad \forall\; r \geqslant 0.
\]	
Hence, there exists a constant $C>0$ such that
	\begin{align*}
	\overline{u}(r)\leqslant Cr^{2(m-1)} \quad \mbox{for any $r\geqslant 1$}.
	\end{align*}
Set $w = \Delta^{m-1} u$. By Lemma~\ref{prow0}(i), there holds $w > 0$. Moreover, as the map $t \mapsto t^\alpha$ is convex in $(0, \infty)$ when $\alpha \leqslant 0$, Jensen's inequality implies
$$-\Delta \overline w = \overline{u^\alpha} \geqslant \overline{u}^\alpha \quad \mbox{in }\; \Rset^n,$$
so that
	\begin{align*}
	-\big(r^{n-1}\overline{w}'(r)\big)' \geqslant r^{n-1}\overline{u}^\alpha(r) \geqslant Cr^{n-1+2(m-1)\alpha}, \quad \forall\; r \geqslant 1.
	\end{align*}
Integrating over $(1,r)$, taking into account $n\geqslant 3$ and $(m-1)\alpha \geqslant -1$, there holds
	\begin{align*}
	\overline{w}'(1)-r^{n-1}\overline{w}'(r)&\geqslant Cr^{2(m-1)\alpha+n}-C.
	\end{align*}
	Therefore,
\begin{equation}\label{for9}
	\overline{w}'(r) \leqslant -Cr^{2(m-1)\alpha+1}+Cr^{-n+1}, \quad \forall\; r \geqslant 1.
\end{equation}
If $\alpha\in(-1/(m-1),0)$, then integrating \eqref{for9} over $[1,r]$ gives
	\begin{align*}
	\overline{w}(r)-\overline{w}(1)\leqslant -Cr^{2(m-1)\alpha+2}+C, \quad \forall\; r \geqslant 1.
	\end{align*}	
We have then $\overline{w}(r)\to -\infty$ as $r\to\infty$, which is a contradiction with $w>0$. If $\alpha=-1/(m-1)$, then integrating of \eqref{for9} over $[1,r]$ gives
	\begin{align*}
	\overline{w}(r)-\overline{w}(1)\leqslant -C\int_{1}^{r}r^{-1}dr+ \int_{1}^{r} Cr^{-n+1}dr =-C\ln r +C,
	\end{align*}
which also implies that $\overline{w}(r)\to -\infty$ as $r\to\infty$. We reach again a contradiction.
\end{proof}


\subsubsection{For $0 \leqslant \alpha \leqslant 1$}

Now we turn to the case of non-negative, sublinear $\alpha$. The following non-existence result is one of the main contributions of this paper.

\begin{proposition}\label{proalpha>0}
For any $n \geqslant 1$, $m\geqslant 1$, and $\alpha \in [0, 1]$, the equation \eqref{negative} has no non-trivial, non-negative solution.
\end{proposition}

\begin{proof}
In view of Proposition \ref{pron=1,2}, it suffices to consider the case $n \geqslant 3$. Depending on the size of $\alpha$, we consider two possible cases. When $\alpha = 0$, the equation \eqref{negative} becomes $\Delta^m u \equiv -1,$ therefore the non-existence of entire, non-negative solution in $\Rset^n$ is a direct consequence of Lemma \ref{lemprioriestimate}, since there exists $C > 0$ such that $u + Cr^{2m}$ is polyharmonic, whose average grows faster than $r^{2m-2}$.

From now on, we only consider $\alpha \in (0, 1]$. For convenience, we divide the proof into three steps.

\smallskip
\noindent\textbf{Step 1}. Suppose that $u$ is a non-trivial, non-negative solution to $\Delta^m u = -u^\alpha$ in $\Rset^n$. By Lemma~\ref{lemprioriestimate}, we have
	\begin{align*}
	\overline u (R) &\leqslant C R^{2(m-1)} \quad \mbox{for }\; R \geqslant 1.
	\end{align*}
Hence
	\begin{align}\label{estFR}
	\int_{B_{2R}} udx =: F(R) \leqslant C R^{n+2(m-1)} \quad \mbox{for any $R \geqslant 1$.}
	\end{align}
Here $C$ is a constant independent of $R$. Note that to get the estimate \eqref{estFR} we only use the sign of $\Delta^m u$. Now via the rescaled test-function argument we fully use the equation $\Delta^m u = -u^\alpha$ to estimate $F(R)$ from below; see \eqref{ualpha}. Let $\psi$ be a smooth cut-off function satisfying $0\leqslant \psi\leqslant 1$ and \eqref{psi}.
	For any $R > 0$, let $$\phi_R(x)=\psi^{2m+1}\left(\frac{x}{R}\right).$$
It is not hard to verify the pointwise estimate $\Delta^m (\psi^{2m+1}) \leqslant C \psi$ for some constant $C>0$. Therefore, we can estimate
\begin{equation}\label{estimate4cutoff}
\begin{aligned}
	|\Delta^m \phi_R(x)|&= R^{-2m}\left|\Delta^m (\psi^{2m+1})\left(\frac{x}{R}\right)\right|\\
	&\leqslant CR^{-2m} \psi\left(\frac{x}{R}\right) = CR^{-2m} \phi_R^{1/(2m+1)}(x).
\end{aligned}
\end{equation}
Hence
	\begin{align*}
	\int_{\Rset^n}u^\alpha \phi_R dx &=-\int_{\Rset^n}\Delta^mu\, \phi_Rdx\\
&=-\int_{\Rset^n} u\,\Delta^m \phi_Rdx \leqslant CR^{-2m} \int_{\Rset^n}u\,\phi_R^{1/(2m+1)}dx.
	\end{align*}
	This yields
	\begin{align}\label{ualpha}
	\int_{B_R}u^\alpha dx\leqslant CR^{-2m}F(R), \quad \forall\; R > 0.
	\end{align}
Now we further examine $F$. In view of \eqref{estFR}, the function $F$ has at most algebraic growth at infinity. Therefore, it must be doubling along a sequence $R_i \to \infty$. We turn this observation into a claim as follows
	\begin{align}\label{Ri}
	\exists \; M > 0 \; \mbox{ and }\; R_i \to \infty \; \mbox{ such that }\; F(2R_i)\leqslant M F(R_i) \quad \forall \; i.
	\end{align}
Indeed, assume that \eqref{Ri} is false. Let us fix $M_0>2^{n+2m-2}$, then there exists $R_0 > 0$ such that
\begin{align}
\label{noRi}
F(2R)\geqslant M_0F(R), \quad \forall\; R\geqslant R_0.
\end{align}
Let $R_1 \geqslant 1$ be sufficiently large verifying $F(R_1)>0$, such a $R_1$ exists since $u$ is non-trivial. Denote $R_*:=\max\{R_1, R_0\}$. Iterating \eqref{noRi} and thanks to \eqref{estFR}, we arrive at, for any $i$,
\begin{align*}
	M_0^iF(R_*)\leqslant F(2^iR_*) \leqslant C (2^iR_*)^{n+2m-2},
	\end{align*}
that is
	\begin{align*}
	\Big(\frac{M_0}{2^{n+2m-2}}\Big)^i \leqslant \frac{CR_*^{n+2m-2}}{F(R_*)} \quad \forall\; i.
\end{align*}
But this is a just impossible if $i$ is large enough by the choice of $M_0$. So the claim \eqref{Ri} holds true. We are now ready to prove the result for $\alpha \in (0,1]$. The two cases $\alpha =1$ and $0< \alpha < 1$ must be considered separately.
	
\smallskip
\noindent\textbf{Step 2}. Consider first the case $\alpha=1$. It follows from \eqref{ualpha} with $R = 2R_i$ and \eqref{Ri} that
	$$F(R_i)\leqslant CR_i^{-2m}F(2R_i) \leqslant CMR_i^{-2m}F(R_i), \quad \forall\; i.$$
Therefore $F(R_i) = 0$ for $i$ large enough because $R_i \to \infty$, which is a contradiction since $u$ is non-trivial.
	
\smallskip
\noindent\textbf{Step 3}. Here we handle the case $\alpha\in (0,1)$. Let $(q_h)$ be the sequence defined as follows
	\begin{align*}
	q_0=1, \;\;\; \frac{1}{q_h}=\frac{\alpha}{q_{h-1}}-\frac{2m}{n}, \quad h=1,2,...
	\end{align*}
By induction, we can compute $q_h$ explicitly as
	\begin{align*}
	\frac{1}{q_h}=\alpha^h-\frac{2m(1-\alpha^h)}{n(1-\alpha)} \quad \mbox{whenever $q_h$ is well defined.}
	\end{align*}
Obviously, the sequence $(q_h^{-1})$ is decreasing since $\alpha \in (0, 1)$, and there exists a unique integer $j_* \geqslant 0$ such that
	$$\frac{1}{q_{j_*+1}}\leqslant 0< q_{j_*}.$$
We will estimate $\|u\|_{L^{q_h}(B_R)}$ successively. First, for all $0 \leqslant h \leqslant j_*$ and $R \geqslant 1$, we claim that
\begin{align}\label{qh_induction}
	\Big(\int_{B_R}u^{q_{h}}dx\Big)^{1/q_h }\leqslant CR^{(n + 2m-2)\alpha^h}.
\end{align}
Firstly, the inequality \eqref{qh_induction} for $h=0$ follows from \eqref{estFR}. Assume that \eqref{qh_induction} is true up to $h-1$ with $h\leqslant j_*$. Using the equation $\Delta^mu = -u^\alpha$, Lemma \ref{lem6}, and
\eqref{estFR}, we get that
	\begin{align}\label{estimate_qh}
	\begin{split}
\Big(\int_{B_R}u^{q_h}dx\Big)^{1/q_h}
	\leqslant & \; CR^{\frac{n}{q_h}+2m-\frac{n\alpha}{q_{h-1}}}\Big(\int_{B_{2R}}|\Delta ^mu|^{q_{h-1}/\alpha}dx\Big)^{\alpha/q_{h-1}}\\
	& +CR^{n/q_h-n}\int_{B_{2R}}udx\\
	 \leqslant &\; CR^{\frac{n}{q_h}+2m-n(\frac{1}{q_h}+\frac{2m}{n})}\Big(\int_{B_{2R}}u^{q_{h-1}}dx\Big)^{\alpha/q_{h-1}} +CR^{n/q_h+2(m-1)}\\
	 = & \; C\Big(\int_{B_{2R}}u^{q_{h-1}}dx\Big)^{\alpha/q_{h-1}} +CR^{n/q_h+2(m-1)}.
	 \end{split}
	\end{align}
Thanks to the induction hypothesis on $\|u\|_{L^{q_{h-1}}(B_R)}$, for any $R \geqslant 1$, it follows from \eqref{estimate_qh} that
 \begin{align*}
 	\Big(\int_{B_R}u^{q_h}dx\Big)^{1/q_h}&\leqslant C\Big(\int_{B_{2R}}u^{q_{h-1}}dx\Big)^\frac{\alpha}{q_{h-1}} +CR^{n/q_h+2(m-1)}\\
 	& = CR^{(n + 2m-2)\alpha^h} +CR^{n/q_h+2(m-1)}\\
 	& \leqslant CR^{(n + 2m-2)\alpha^h}.
 \end{align*}
 For the last line, we have used
 \begin{align*}
 (n + 2m-2)\alpha^h - \Big(\frac{n}{q_h}+2m-2\Big) = (1 - \alpha^h) \Big[\frac{2m}{1 - \alpha}- 2(m-1)\Big] > 0
 \end{align*}
to absorb the term $R^{n/q_h+2(m-1)}$ into the term $R^{(n + 2m-2)\alpha^h}$. Hence the claim \eqref{qh_induction} holds true for all $h\leqslant j_*$ and $R\geqslant 1$. Furthermore, by the definition of $j_*$, there holds
$$\frac{\alpha}{q_{j_*}}- \frac{2m}{n}\leqslant 0.$$
Therefore, for any $q>1$, the condition \eqref{exponentq} is always fulfilled with $\ell = q_{j_*}/\alpha$. Applying Lemma \ref{lem6}, \eqref{qh_induction} with
$h = j_*$, and \eqref{estFR}, we get
	\begin{align}
\label{uj}
\begin{split}
	\Big(\int_{B_R}u^{q}dx\Big)^{1/q}&\leqslant CR^{\frac{n}{q}+2m-\frac{n\alpha}{q_{j_*}}}\Big(\int_{B_{2R}}|\Delta ^mu|^{q_{j_*}/\alpha}dx\Big)^{\alpha/q_{j_*}}\\
& \quad +CR^{n/q-n}\int_{B_{2R}}udx\\
	&\leqslant CR^{\frac{n}{q}+2m-\frac{n\alpha}{q_{j_*}}}\Big(\int_{B_{2R}}u^{q_{j_*}}dx\Big)^{\alpha/q_{j_*}} +CR^{\frac{n}{q}+2(m-1)}\\
	&\leqslant CR^{\frac{n}{q}+2m-\frac{n\alpha}{q_{j_*}}} R^{(n+ 2m-2) \alpha^{j_*+1}} +CR^{\frac{n}{q}+2(m-1)}\\
	& = CR^{\frac{n}{q}-2\alpha^{j_*+1}+ 2m\frac{1-\alpha^{j_*+2}}{1-\alpha}} +CR^{\frac{n}{q}+2(m-1)}\\
	&\leqslant CR^{\frac{n}{q}-2\alpha^{j_*+1}+ 2m\frac{1-\alpha^{j_*+2}}{1-\alpha}}.
\end{split}
	\end{align}
Keep in mind that $\alpha < 1 < q$. In the following step, we aim to estimate $\int_{B_{2R}}u dx$ from above by using $\int_{B_{2R}}u^\alpha dx$ and $\int_{B_{2R}}u^{q}dx$. On the other hand, let
$$a=\alpha\frac{q-1}{q - \alpha} \in [0, 1), \quad p= \frac{q-\alpha}{q-1} > 1$$
and apply H\"older's inequality with help from \eqref{ualpha} to obtain
	\begin{align}\label{forF}
	 \begin{split}
	 F(R) = \int_{B_{2R}} u dx &\leqslant \Big(\int_{B_{2R}} u^{ap}dx\Big)^\frac{1}{p} \Big(\int_{B_{2R}} u^{(1-a)\frac{p}{p-1}}dx\Big)^\frac{p-1}{p}\\
	 & = \Big(\int_{B_{2R}} u^{\alpha}dx\Big)^\frac{q-1}{q - \alpha}\, \Big(\int_{B_{2R}} u^{q}dx\Big)^\frac{1-\alpha}{q - \alpha}\\
& \leqslant C \big[R^{-2m}F(2R)\big]^\frac{q-1}{q - \alpha}\, \Big(\int_{B_{2R}} u^{q}dx\Big)^\frac{1-\alpha}{q - \alpha}.
	 \end{split}
	\end{align}
Now we apply \eqref{forF} for the sequence $(R_i)$ satisfying the doubling property \eqref{Ri} to get
	\begin{align*}
	F(R_i)\leqslant C R_i^{-\frac{2m(q-1)}{q-\alpha}}F(R_i)^\frac{q-1}{q - \alpha}\Big(\int_{B_{2R_i}} u^{q}dx\Big)^\frac{1-\alpha}{q - \alpha}.
	\end{align*}
Combining the above estimate with \eqref{uj}, we get
	\begin{align}
\label{estFRi}
\begin{split}
	F(R_i) & \leqslant C R_i^{-\frac{2m(q-1)}{1-\alpha}}\int_{B_{2R_i}} u^qdx\\
 &\leqslant C R_i^{-\frac{2m(q-1)}{1-\alpha}}\, (2R_i)^{n + \big[-2\alpha^{j_*+1}+ 2m\frac{1-\alpha^{j_*+2}}{1-\alpha}\big]q}\\
	&= C (2R_i)^{\frac{2m}{1-\alpha}+n- \big(\frac{2m\alpha}{1-\alpha} + 2\big)\alpha^{j_*+1}q}.
\end{split}
	\end{align}
We fix $q > 1$ large enough such that
$$\frac{2m}{1-\alpha}+n- \Big(\frac{2m\alpha}{1-\alpha} + 2\Big)\alpha^{j_*+1}q < 0.$$
Then the estimate \eqref{estFRi} implies that $F(R_i)\to 0 $ as $i\to \infty$. Immediately, this is a contradiction because $u$ is non-trivial.
\end{proof}

\subsection{Non-existence results for $(-\Delta)^m u = -u^\alpha$ with $\alpha > 1$}
\label{subnew1}

In this subsection, we consider non-negative classical solutions of
\begin{align}
\label{new3}
(-\Delta)^{m} u = -u^\alpha \quad \mbox{in }\;\Rset^n
\end{align}
under the restriction $\alpha > 1$. This corresponds to the last column of Tables \ref{tabevenminus} and \ref{tabevenplus}.

\begin{proposition}\label{proalpha>1evenminus}
For any $n \geqslant 1$, $m \geqslant 1$, and $\alpha > 1$, the equation \eqref{new3} has no non-trivial non-negative solution.
\end{proposition}

\begin{proof}
Assume that $u$ is a non-negative solution in $\Rset^n$ of \eqref{new3} with $\alpha > 1$.
We first derive an integral estimate of $u$ over $B_R$. Let $\psi$ be a smooth, radial, cut-off function satisfying
$0\leqslant \psi\leqslant 1$ and \eqref{psi}. For any $R>0$, let
$$\phi_R(x)=\psi^p(R^{-1}x)$$
with $p=\frac{2m\alpha}{\alpha-1} > 2m$. By mimicking the argument leading to \eqref{estimate4cutoff} we know that the pointwise estimate $|\Delta^m (\psi^p)| \leqslant C \psi^{p-2m}$ holds for some constant $C>0$. Hence, we eventually have
	\begin{align*}
		|\Delta^{m} \phi_R(x)| \leqslant CR^{-2m} \psi^{p-2m}\left(\frac{x}{R}\right) = CR^{-2m} \phi_R^{1/\alpha}(x).
	\end{align*}
Therefore, similar to the estimate after \eqref{estimate4cutoff}, we obtain the following
\begin{equation} \label{test1}
\begin{aligned}
		\int_{\Rset^n}u^\alpha \phi_R dx &= -\int_{\Rset^n}u(-\Delta)^{m} \phi_Rdx \leqslant CR^{-2m} \int_{\Rset^n}u\phi_R^{1/\alpha}dx.
\end{aligned}
\end{equation}
Using H\"older's inequality and noting the support of $\phi_R$,
\begin{align}\label{test2}
	\int_{\Rset^n}u\phi_R^{1/\alpha}dx \leqslant CR^{\frac{n(\alpha-1)}{\alpha}}\Big(\int_{\Rset^n}u^\alpha\,\phi_Rdx\Big)^{1/\alpha}.
\end{align}
Putting together \eqref{test1} and \eqref{test2}, we arrive at
\begin{align*}
	\int_{B_R}udx \leqslant \int_{\Rset^n}u\phi_R^{1/\alpha}dx \leqslant CR^{n-\frac{2m}{\alpha-1}}.
\end{align*}
Applying Lemma~\ref{lem_general_Liouville}, we deduce that $u\equiv 0$ in $\Rset^n$.
\end{proof}



\subsection{Existence results for $\Delta^m u=u^\alpha$}
\label{subnew2}

This subsection is devoted to the existence results for the equation
\begin{align}
\label{new4}
\Delta^{m} u = u^\alpha \quad \mbox{in }\; \Rset^n
\end{align}
under the condition $\alpha \leqslant 1$, in other words, we do not consider the situations under applications of Propositions \ref{propA} and \ref{propB}.

\begin{proposition}\label{proalpha<1plus}
	For any positive integers $m$, $n$ and $\alpha \leqslant 1$, the equation \eqref{new4} has infinitely many positive radial solutions.
\end{proposition}

\begin{proof}
We look for radial solutions of \eqref{new4}. To this purpose, consider the following initial value problem
\begin{equation}\label{eqradial}
\left\{
\begin{split}
&\Delta^m u(r) = u^\alpha(r) && \text{ in } [0, R),\\
&u(0) = 1;\; \Delta^i u(0)= a_i > 0, &&1 \leqslant i \leqslant m-1.
\end{split}
\right.
\end{equation}
Clearly, using standard ODE theory, \eqref{eqradial} has a unique positive solution in a maximal interval $[0, R_{\max})$. To turn that solution into an entire solution to \eqref{new4}, we need only to prove that $R_{\max}=\infty$. To do this, first we construct suitable sub- and super-solutions to \eqref{eqradial}, which are defined to all time, then apply then the comparison principle.

Indeed, let $u_* (r) \equiv 1$. Trivially, $u_*$ is a sub-solution to \eqref{eqradial} and $\Delta^iu_*(0) \leqslant \Delta^i u(0)$ for any
$0 \leqslant i \leqslant m-1$. Hence $u(r) \geqslant 1 = u_*(r)$ in $[0, R_{\max})$ by Lemma \ref{comparison} with $f \equiv 0$, which is clearly non-decreasing.

Now we turn our attention to the existence of a super-solution. Let $v(r) = e^{r^2/2}$ and denote
$$\Delta^k v(r) =P_k(r)e^{r^2/2}. $$
A direct computation yields $P_1(r)=r^2+n$ and $P_2 (r)=r^4 + (2n+4)r^2 + n^2+2n$; in fact, there holds
$$P_{k+1}(r)=(r^2+n)P_k+ 2rP'_k(r)+\Delta P_k(r) \quad \forall\; k \geqslant 0, \; r\geqslant 0.$$
By induction, we can readily prove that each $P_k$ is a polynomial whose coefficients are natural numbers and ${\rm deg}(P_k)=2k$. Because
\[
P_{k+1}(0)= n P_k (0) + \Delta P_k (0) = n P_k (0) + n P''_k (0)
\]
we can easily see that $P_k(0) \geqslant 1$ for all $k$. Now we let $u^*(r) = \lambda v(r)$ with $\lambda = \max(1, a_1, \ldots a_{m-1}) \geqslant 1$. It follows that
\begin{align}
\label{supsolution}
\Delta^m u^* (r) = \lambda P_m(r)e^{r^2/2} \geqslant \lambda P_m(0)e^{r^2/2} \geqslant \lambda e^{r^2/2}, \quad \forall\; r \geqslant 0
\end{align}
and $\Delta^iu^*(0) = \lambda P_i(0) \geqslant \lambda \geqslant \Delta^i u(0)$ for $0 \leqslant i \leqslant m-1$. There are two possibilities:

\noindent{\sl Case 1: $\alpha \in [0, 1]$}. In this case, \eqref{supsolution} yields $\Delta^m u^* (r) \geqslant  \lambda^\alpha e^{\alpha r^2/2} = u^* (r)^\alpha$ for $r \geqslant 0$. Therefore, $u^*$ is indeed a super-solution to \eqref{eqradial}. From this we can apply Lemma \ref{comparison} with $f(t) = t^\alpha$ in $\Rset_+$ to get that $u(r) \leqslant u^*(r)$ in $[0, R_{\max})$.

\noindent{\sl Case 2: $\alpha <0$}. In this case, on one hand, \eqref{supsolution} yields $\Delta^m u^* \geqslant 1$. On the other hand, we have already shown that $u \geqslant 1$ in $[0, R_{\max})$, which immediately yields $1 \geqslant u^\alpha = \Delta^m u$ in $B_{R_{\max}}$. Therefore, this time applying Lemma \ref{comparison} with $f \equiv 1$, now there holds $u(r) \leqslant u^*(r)$ in $[0, R_{\max})$.

Combining two cases, whenever $\alpha \leqslant 1$ we can always conclude that $u_* \leqslant u \leqslant u^*$ in $B_{R_{\max}}$, where $u_*$ and $u^*$ are smooth positive functions in $\Rset^n$ constructed above. As $u$ is locally uniformly bounded, we can obtain the local boundedness of all derivatives of $u$ up to order $2m-1$ by successive integrations; see \cite[Proposition A.2]{FF16}. This readily implies that $R_{\max} = \infty$ as claimed.

The infinity of solutions can be obtained by choosing different values of $a_i$ if $m\geqslant 2$; or at least by the natural scaling of the equation \eqref{new4}.
\end{proof}

\begin{remark}
If $n \geqslant 2$, we can also put solutions of lower dimensions in $\Rset^n$, to get infinitely many non radial solutions of \eqref{new4} with $\alpha \leqslant 1$. Similar remark goes to the equation \eqref{negative} when $n \geqslant 4$.
\end{remark}


\section{Maximum principle type result}
\label{sec-Comparision}
This subsection is devoted to proofs of Propositions \ref{thmCompareSetMinus} and \ref{thmCompareSetPlus}. We use an elementary property for non-negative super-polynharmonic radial functions.
\begin{lemma}\label{chanduoiw2}
Let $m\geqslant 1$. Assume that $w$ is non-trivial, non-negative, radial function such that $(-\Delta)^m w \geqslant 0$ in $\Rset^n$. Then, either $w(0)>0$ or $\lim_{r \to +\infty} w(r) = +\infty$.
\end{lemma}

\begin{proof}
Let $m=1$, if $w$ is non-negative, non-trivial, and $-\Delta w \geqslant 0$, then the strong maximum principle yields $w(0) > 0$. Suppose that the conclusion is true up to some positive integer $m$, we consider now $w$ such that $(-\Delta)^{m+1} w \geqslant 0$ in $\Rset^n$. We have two cases:

\noindent\textit{Case 1}: If $m+1$ is odd, then by using Lemma \ref{prow0}(ii), there holds $\Delta^m w\geqslant 0$, namely $(-\Delta)^m w\geqslant 0$ in $\Rset^n$. From this we get the result by induction hypothesis.

\noindent\textit{Case 2}: If $m+1$ is even, then we have $\Delta^{m+1}w \geqslant 0$, which means that $\Delta^m w(r)$ is non-decreasing in $r$. Therefore, there exists
\[
\lim_{r \to +\infty} \Delta^m w(r) = \ell \in \Rset \cup \{ + \infty \}.
\]
If $\ell > 0$, then we readily have $\lim_{r \to +\infty} w(r) = +\infty$ by comparison principle. If $\ell \leqslant 0$, then there holds $\Delta^m w \leqslant 0$ everywhere, namely $(-\Delta)^m w\geqslant 0$ in $\Rset^n$. Again we conclude with the induction hypothesis.

Combing the above two cases, we are done.
\end{proof}

\subsection{A maximum principle type result for $\Delta^m u =- u^\alpha$}

We now in position to prove Proposition \ref{thmCompareSetMinus}. Let $u$ be a non-trivial, non-negative solution to $\Delta^m u =- u^\alpha$. In view of Theorem \ref{thmMinus} we are limited to the case where $m$ is odd with $\alpha \geqslant \ps(m)$.

By way of contradiction, suppose that $u(x_0) = 0$ for $x_0\in \Rset^n$. Without loss of generality, we can assume that $x_0 =0$. Following the proof of Proposition \ref{proalpha>1evenminus}, there holds
\begin{equation}\label{eq:priorest}
\int_{B_R} u dx \leqslant C R^{n- \frac{2m}{\alpha -1}}.
\end{equation}
(Keep in mind that to get \eqref{eq:priorest}, we need only $\alpha>1$.) Taking the average over spheres, we get $\overline u(0) =0$, and more importantly
$$\Delta^m \overline u = -\overline {u^\alpha} \leqslant 0 \quad \mbox{in }\; \Rset^n.$$
Since $u$ is non-trivial and non-negative, so is $\overline u$. Applying Lemma \ref{chanduoiw2}, as $\overline u(0) =0$, there exists $r_0 >0$ such that $\overline u(r) \geqslant 1$ for any $r \geqslant r_0$. Consequently, we obtain a contradiction with \eqref{eq:priorest} if $R\to \infty$. \qed

\subsection{A maximum principle type result for $\Delta^m u = u^\alpha$}
Now we consider $\Delta^m u = u^\alpha$ with $\alpha \geqslant 0$. In view of Theorem \ref{thmPlus}, we are limited to the case either $0 \leqslant \alpha \leqslant 1$ or $m$ is even and $\alpha \geqslant \ps(m)$.

Firstly, for the case $0\leqslant \alpha \leqslant 1$ and $m \geqslant 2$, we can construct easily non-trivial non-negative radial solutions with $u(0) = 0$. Indeed, consider the similar initial value problem to \eqref{eqradial} where we choose
$$u(0) = 0, \quad \Delta u(0) = a_1 > 0, \quad \Delta^i u(0) = a_i \geqslant 0, \; i = 2,\ldots m-1.$$
Then we get a global solution which is non-trivial because it verifies
\[
u(r) \geqslant u_*(r):= \frac{a_1}{2n}r^2 \quad \text{ in } \; \Rset^n.
\]
Another easy fact is that when $\alpha \in [0, 1)$ and $m \geqslant 1$. As $\frac{2m}{1-\alpha} \geqslant 2m$, there exists $C > 0$ depending on $m$, $n$, and $\alpha$ such that $Cr^\frac{2m}{1-\alpha}$ is an entire classical solution for $\Delta^m u = u^\alpha$ in $\Rset^n$. Remark that the function $r^\frac{2m}{1-\alpha}$ belongs to $C^{2m}(\Rset^n)$.

Consider now the case $\alpha >1$ and $m$ is {\sl even}. The proof, based on a contradiction argument, is very similar to the  $\Delta^m u = -u^\alpha$ case with $m$ odd. Let $u$ be a non-trivial non-negative solution verifying $u(0) =0$, then the estimate \eqref{eq:priorest} remains valid. As $(-\Delta)^m \overline u \geqslant 0$ in $\Rset^n$ and $\overline u(0) = 0$, by Lemma \ref{chanduoiw2}, there exists $r_0 >0$ such that $\overline u(r) \geqslant 1$ for $r$ large, which is impossible seeing \eqref{eq:priorest}.

It remains to consider $m=\alpha=1$. Suppose that a classical non-negative solution to $\Delta u = u$ exists in $\Rset^n$, and $u(x_0) = 0$. We can assume that $x_0 = 0$. Clearly $\overline u$ is still a classical solution to the same equation, hence $\overline u'(0) = 0$ and $r\mapsto \overline u(r)$ is nondecreasing. By direct integration
$$\overline u'(r) = r^{1-n}\int_0^r s^{n-1}\overline u(s)ds \leqslant \frac{r}{n}\overline u(r), \quad \forall\; r \geqslant 0.$$
From this and $\overline u(0)=0$, the classical Gronwall estimate yields then $\overline u \equiv 0$ in $\Rset^n$, so is $u$, which is absurd. This completes our proof. \qed


\section*{Acknowledgments}

This work was initiated while QAN was visiting the Vietnam Institute for Advanced Study in Mathematics (VIASM) in 2017. He gratefully acknowledges the institute for its hospitality and support. The research of QAN is funded by Vietnam National Foundation for Science and Technology Development (NAFOSTED) under grant number 101.02-2016.02 and supported by the Tosio Kato Fellowship. The research of VHN is funded by the Simons Foundation Grant Targeted for Institute of Mathematics, Vietnam Academy of Science and Technology. The research of QHP is funded by Vietnam National Foundation for Science and Technology Development (NAFOSTED) under grant number 101.02-2017.307. DY is partially supported by Science and Technology Commission of Shanghai Municipality (STCSM) under grant No. 18dz2271000.

\section*{ORCID IDs}

\noindent Qu\cfac oc Anh Ng\^o: 0000-0002-3550-9689, Van Hoang Nguyen: 0000-0002-0030-5811

\end{document}